\documentclass[a4paper, 11pt]{amsart}
\usepackage[initials]{amsrefs}
\usepackage{amssymb}
\usepackage{stmaryrd}
\usepackage[margin=2.5cm]{geometry}
\usepackage[usenames]{xcolor}
\usepackage{pdflscape}

\newcommand{\bM}{\overline{M}}
\newcommand{\cM}{\mathcal{M}}

\newcommand{\cL}{\mathfrak{L}}
\newcommand{\pp}{\mathbb{P}}
\newcommand{\cc}{\mathbb{C}}
\renewcommand{\qq}{\mathbb{Q}}
\newcommand{\zz}{\mathbb{Z}}
\newcommand{\nn}{\mathbb{N}}
\renewcommand{\ss}{\mathbb{S}}
\newcommand{\pgl}{\mathbf{PGL}}

\newcommand{\cV}{\mathcal{V}}
\newcommand{\fm}{\mathfrak{m}}
\newcommand{\cW}{\mathcal{W}}
\newcommand{\cP}{\mathcal{P}}
\newcommand{\cG}{\mathcal{G}}
\newcommand{\cH}{\mathcal{H}}
\newcommand{\tT}{\mathbb{T}}
\newcommand{\ud}[2]{_{#1}^{(#2)}}
\newcommand{\lp}{\left(}
\newcommand{\rp}{\right)}

\newtheorem{theorem}{Theorem}[section]
\newtheorem{proposition}[theorem]{Proposition}

\theoremstyle{remark}
\newtheorem*{remark}{Remark}

\DeclareMathOperator{\ch}{ch}
\DeclareMathOperator{\D}{\mathbf{D}}
\DeclareMathOperator{\fix}{fix}
\DeclareMathOperator{\Tr}{Tr}
\DeclareMathOperator{\rk}{rk}

\title{Equivariant Cohomology of certain Moduli of Weighted Pointed 
Rational Curves}
\author{Chitrabhanu Chaudhuri}

\begin{document}
\maketitle 

\section{Introduction}
In \cite{Has} Hassett introduces and studies the moduli spaces of weighted pointed stable
curves. A weighted pointed curve is a nodal curve with a sequence of smooth marked points,
each assigned a rational number between 0 and 1. A subset of the  marked points may 
coincide if the sum of their weights is at most 1. 

The moduli spaces are connected, smooth and proper Deligne-Mumford stacks. In the special 
case of genus zero the moduli spaces are smooth projective varieties. Throughout this paper
we work over $\cc$ as the base field and we always consider cohomology with $\cc$ coefficients.

Consider the weight data 
\begin{equation} \label{eq:wtdata}
  \mathcal{A}(m,n) = \left(\underbrace{1,\ldots,1}_{m}, \underbrace{1/n,\ldots,1/n}_{n}
  \right) \quad m+n\geq 3,\  m \geq 2 \ .
\end{equation} 
Let 
\begin{equation*}
  \bM_{0,m|n} = \bM_{0, \mathcal{A}(m,n)} \ .
\end{equation*}
$\bM_{0,m|n}$ parametrises nodal curves with $m+n$ smooth marked points such that
the first $m$ marked points are distinct but any subset of the last $n$ marked points
can coincide. There is naturally an action of $S_m \times S_n$ on $\bM_{0,m|n}$. Here 
$S_m$ permutes the first $m$ marked points and $S_n$ permutes the last $n$. 

In this paper we study the induced action of $S_m\times S_n$ on the cohomology of $\bM_{0,m|n}$ 
and calculate the equivariant Poincar\'{e} polynomial for some small values of $m$ and $n$. 

Let $M_{0,m|n}$ be the interior of the the moduli space, parametrizing only the smooth curves. 
We first derive the $S_m\times S_n$ character on $H^*(M_{0,m|n})$, and write down a generating 
function for the characters. We then describe a recipe for calculating the generating function
for the $S_m\times S_n$ character of $H^*(\bM_{0,m|n})$. This is achieved by analysing a 
spectral sequence relating the cohomology of $M_{0,m|n}$ to that of $\bM_{0,m|n}$. 

It should be noted that when $n=0$, we simply get the moduli of stable rational curves with
$m$ marked points. In this case the equivariant cohomology was studied by Getzler 
\cite{Get}.

In another direction when $m=2$, the moduli spaces under consideration are the 
Losev- Manin spaces of \cite{LM}. The $S_2\times S_n$ action on the cohomology  was 
determined in this case by Bergstr\"{o}m and Minabe \cite{BM2}.

Finally Bergstrom and Minabe \cite{BM1} give a recursive method for calculating the 
equivariant Poincar\'{e} polynomial of $\bM_{0,m|n}$ for all $m$ and $n$. However our
method seems more direct. We use techniques developed by Getzler \cite{Get} and 
Getzler and Kapranov \cite{GK}. We adopt the notation $\bM_{0, m|n}$ from \cite{BM2}.

\textbf{Aknowledgements.} The author is grateful for fruitful discussions with Prasit 
Bhattacharya, Ezra Getzler and Ronnie Sebastian. A major part of the work was done while 
the author was a postdoctoral fellow at the Max Planck Institute for Mathematics, Bonn. 

\section{Preliminaries on $\bM_{0,m|n}$}
Let $\mathcal{A}(m,n)$ be as in \eqref{eq:wtdata}. Following Hassett \cite{Has},  
$\bM_{0,m|n}$ is the moduli of weighted pointed stable curves of of genus zero 
corresponding to the weight data $\mathcal{A}(m,n)$. 

When $m>3$, $\bM_{0,m|0}$ is simply the moduli of stable rational curves with 
$m$ marked points. We abbreviate it as $\bM_{0,m}$.

Denote by $M_{0,m|n}$ the open subvariety parametrising the smooth curves. 

\subsection{The stable curves}
An $\mathcal{A}(m,n)$-stable curve $(C; p_1,\ldots,p_{m+n})$ is a nodal curve with 
smooth marked points $p_i$. The marked points of $C$ along with the nodes will be called 
special points.  We shall call the first $m$ marked points along with the nodes special 
points of type 1, whereas the last $n$ marked points will be referred to as special 
points of type 2. The curve $C$ satisfies the following,
\begin{itemize}
\item Arithmetic genus of $C$ is $0$.
\item The points $\{p_1,\ldots,p_m\}$ are all distinct.
\item Any subset of the points $\{ p_{m+1},\ldots, p_{m+n} \}$ can coincide, but these
      points are all distinct from $\{p_1,\ldots,p_m\}$.
\item Any irreducible component of $C$ has at least 3 special points with at least 2 of 
      type 1.
\end{itemize}

The varieties $\bM_{0,m|n}$ are smooth and projective and $\bM_{0,m|n} \backslash 
M_{0,m|n}$ is a divisor with normal crossings. Ceyhan \cite{Cey} studies the 
cohomology of $\bM_{0,\mathbf{A}}$, for any weight data $\mathbf{A}$. As a special case 
it follows that all the cohomology of $\bM_{0,m|n}$ is algebraic. This means that all 
the odd degree cohomology groups vanish and the even cohomology groups are isomorphic 
to the Chow groups
\begin{equation*} 
  H^{2i+1}(\bM_{0,m|n}, \qq) =0 \quad \text{and} 
  \quad H^{2i}(\bM_{0,m|n}, \qq) \cong A^{i}(\bM_{0,m|n},\qq).
\end{equation*}   

\subsection{Dual graphs}
A graph will be a triple $(F, V, \sigma)$. Where 
\begin{enumerate}
\item $F$ is the set of flags;
\item $V$ is a partition of $F$; 
\item $\sigma$ is an involution on $F$.
\end{enumerate}
The parts of $V$ are the vertices of the graph. For $v\in V$, let $F(v) = 
\{f\in v \mid f\in F\}$ be the flags incident on $v$. The fixed points of $\sigma$ are
the leaves. The set of leaves will be denoted by $L$, and those incident on a vertex 
$v$ denoted as $L(v)$. The two cycles of $\sigma$ will be the edges of the graph and the 
set of edges denoted by $E$.

\textbf{Colouring of a graph} $G$ consists of a set $X$ and a function $c: F(G) \to 
X$ such that $c(f) = c(\sigma f)$ for every flag $f$. A colouring assigns a colour
(an element of $X$) to each flag such that both flags of an edge have the same colour. 
It thus makes sense to talk about the colour of an edge. 

\textbf{Geometric realisation} of a graph $G$, denoted by $|G|$ is a topological space.
It is the quotient space of, the collection of intervals indexed by the flags of $G$, by 
an equivalence relation.
\[|G| = \frac{F(G)\times [0,1]}{\sim}\]
Here $(f_1,0) \sim (f_2,0)$ if the flags $f_1, f_2$ are incident on the same vertex
and $(f,1) \sim (f',1)$ if the flags $f,f'$ are part of an edge. 

A tree $T$ is a graph such that $|T|$ is connected and simply connected. 

The dual graph of an $\mathcal{A}(m,n)$-stable curve is a tree coloured by $\{1,2\}$. 
The tree has one flag for each marked point and two for each node. For every 
irreduclible component it has a vertex. The marked points correspond to the leaves and 
the nodes correspond to the edges. The flags corresponding to the special points of type 
1 have colour 1 where as the flags corresponding to the special points of type 2 have 
colour $2$. Further the leaves are numbered $1$ to $m+n$ according to the marked point 
it represents. 

\subsection{Strata of $\bM_{0,m|n}$}
It is clear that the dual graphs of $\mathcal{A}(m,n)$-stable curves have to satisfy 
certain constraints. Let $T$ be such a dual graph. For any vertex $v \in V(T)$ let  
$F_1(v)$ be the flags of colour $1$ and $F_2(v)$ the flags of colour $2$. Then we
must have $|F(v)| \geq 3$ and $|F_1(v)| \geq 2$. Let us call such trees 
$\mathcal{A}(m,n)$-stable and denote the isomorphism classes of such trees by 
$\tT(m,n)$. 

For any $T \in \tT(m,n)$ let $M(T)$ be the subvariety of $\bM_{0,m|n}$ parametrising
curves whose dual graphs are isomorphic to $T$. Let $\bM(T)$ be the closure. It is 
clear that (see Ceyhan \cite[Section~3]{Cey})
\begin{equation*}
  M(T) \cong \prod_{v \in V(T)} M_{0, \# F_1(v)\mid \# F_2(v)} \quad \text{and} \quad
  \bM(T) \cong \prod_{v \in V(T)} \bM_{0, \# F_1(v)\mid \# F_2(v)} .
\end{equation*}
The codimension of $\bM(T)$ is equal to the number of edges $|E(T)|$ of $T$.

We have a stratification by dual graphs
\begin{equation*}
  \bM_{0,m|n} = \bigsqcup_{T \in \tT(m,n)} M(T).
\end{equation*}

\section{Symmetric Group Representations}

\subsection{Symmetric functions}.
For results and notation of this section we refer to Macdonald \cite{Mac}. Let $\Lambda 
= \displaystyle \lim_{\leftarrow} \zz \llbracket x_1,\ldots, x_n \rrbracket^{S_n}$ be 
the ring of symmetric functions. It is well known that 
\begin{equation*}
  \Lambda \otimes \qq = \qq \llbracket p_1, p_2, \ldots \rrbracket
\end{equation*}
where $p_k = \sum_{i=1}^{\infty} x_i^k$ are the power sums. Let $\lambda = 
(\lambda_1,\dots, \lambda_k)$ be a partition of $n$, which we denote by $\lambda \vdash 
n$; define $p_{\lambda} = p_{\lambda_1}\cdots p_{\lambda_k}$. For an $S_n$ representation
$V$ we define the symmetric function 
\begin{equation*}
  \ch_n(V) = \frac{1}{n!} \sum_{\sigma \in S_n} \Tr_V(\sigma)
  p_{\lambda(\sigma)} \ ,
\end{equation*}
here $\lambda(\sigma)$ is the partition corresponding to the cycle decomposition of 
$\sigma$. 

The irreducible representations of $S_n$ are indexed by partitions of $n$. For 
$\lambda \vdash n$ let $V_{\lambda}$ be the corresponding irreducible representation. 
The Schur functions also indexed by partitions of $n$ are defined as
\begin{equation*}
  s_{\lambda} = \ch_n(V_{\lambda}) .
\end{equation*}
Schur functions $\{s_{\lambda} \mid \lambda \vdash n, n \geq 1 \}$ form an additive
basis of $\Lambda$. There are also the elementary symmetric functions $e_n = s_{1^n}$ 
and the complete symmetric functions $h_n = s_n$.

There is an associative product $\circ$ on $\Lambda$ called plethysm. It is 
characterised by the fact that 
\begin{equation*}
  \ch_n \lp \mathrm{Ind}_{S_k \wr S_n}^{S_{kn}} V_1 \boxtimes V_2 \boxtimes \cdots
  \boxtimes V_2 \rp  = \ch_k(V_1) \circ \ch_n(V_2) \ ,
\end{equation*}
where $S_k \wr S_n$ is the wreath product $S_k \ltimes (S_n)^k$, $V_1$ is a representation
of $S_k$ and $V_2$ is a representation of $S_n$. 

Let $\Lambda^{(2)} =  \Lambda \otimes \Lambda$. We denote the symmetric 
functions in the first tensor factor by the superscript $(1)$ and those in the second 
tensor factor by the superscript $(2)$.

For $V$ a representation of $S_m\times S_n$ we define 
\begin{equation*}
  \ch_{m|n} (V) = \frac{1}{m!\times n!} \sum_{(\sigma,\tau) \in S_m\times S_n} 
  \Tr_{V}(\sigma,\tau) p\ud{\lambda(\sigma)}{1} p\ud{\lambda(\tau)}{2} 
  \in \Lambda^{(2)} .
\end{equation*}

We shall need the following result later on.
\begin{proposition} \label{prop:operator}
  Let $W$ be any representation of $S_n$ and $\D$ the following  differential 
  operator on $\Lambda$
  \begin{equation*}
    \D = p_1 \frac{\partial}{\partial p_1} -1 \ ,
  \end{equation*}
  then 
  \begin{equation*}
    \ch_n \lp W \otimes V_{(n-1,1)} \rp = \D \ch_n(W) \ .
  \end{equation*}
  $V_{(n-1,1)}$ is the irreducible representation corresponding to the partition 
  $(n-1,1)$ and often referred to as the standard representation of $S_n$.
\end{proposition}

\begin{proof}
  Let $\fix(\sigma)$ denote the number of fixed points of $\sigma \in S_n$. 
  Recall that $\mathrm{Tr}_{V_{(n-1,1)}}(\sigma) = \fix(\sigma) -1$. Also 
  note that $\lambda(\sigma) = (1^{\fix(\sigma)},2^{a_2}, \ldots)$; so $p_{\lambda
  (\sigma)} = p_1^{\fix(\sigma)}p_2^{a_2}\cdots p_n^{a_n}$. Thus 
  \begin{equation*}
    p_1\frac{\partial p_{\lambda(\sigma)}}{\partial p_1} = \fix(\sigma)  
    p_{\lambda(\sigma)} \ . 
  \end{equation*}
  Hence
  \begin{align*}
    \ch_n \lp W \otimes V_{(n-1,1)} \rp & = \frac{1}{n!} \sum_{\sigma \in S_n}
     (\fix(\sigma)-1) \Tr_W(\sigma) p_{\lambda(\sigma)} \\     
    & =   \frac{1}{n!} \sum_{\sigma \in S_n} \Tr_W(\sigma) \lp 
    p_1\frac{\partial p_{\lambda(\sigma)}}{\partial p_1} - p_{\lambda(\sigma)} \rp 
    = p_1\frac{\partial \ch_n(W)}{\partial p_1}  - \ch_n(W) \ .
  \end{align*}
\end{proof}

\subsection{$\ss$ modules} 
\label{sec:smod}
An $\ss$ module (as in \cite[\S 1]{Get}) $\cV$ is a sequence of graded vector spaces 
$\{\cV(n) \mid n \in \nn\}$ with an action of $S_n$ on $\cV(n)$. The characteristic of 
an $\ss$ module is defined as a symmetric series in $\Lambda\llbracket t \rrbracket$ 
\begin{equation*}
 \ch_t(\cV) = \sum_{n=1}^{\infty} \sum_{i\in \zz} (-t)^i\ch_n(\cV^i(n)) .
\end{equation*} 
Here $\cV^i(n)$ is the $i$-th graded component of $\cV(n)$. 

Similarly an $\ss^2$ module $\cW$ is a collection of graded vector spaces $\{\cW(m,n)
\mid (m,n) \in \nn^2\}$, with an action of $S_m\times S_n$ on $\cW(m,n)$. We define 
the characteristic in an analogous way
\begin{equation*}
  \ch_t(\cW) = \sum_{m=1}^{\infty}\sum_{n=1}^{\infty} \sum_{i\in \zz} 
  (-t)^i\ch_{m|n}(\cW^i(m,n)) \in \Lambda^{(2)}\llbracket t \rrbracket \ .
\end{equation*}

In the case of an ungraded $\ss^2$ module $\cW$ we write the characteristic as 
$\ch(\cW)$. We define the $\ss^2$ module $\tT\cW$ in the following way
\begin{equation*}
  \tT\cW(m,n) = \bigoplus_{T \in \tT(m,n)} \cW(T)\ .
\end{equation*}
Here $\tT(m,n)$ are the isomorphism classes of $\mathcal{A}(m,n)$ stable trees and
\begin{equation*}
  \cW(T) = \bigotimes_{v \in V(T)} \cW(F(v)) \ .
\end{equation*}
For a more detailed discussion see Getzler and Kapranov \cite{GK}.

\subsection{Partial Legendre transform}
Let $\rk : \Lambda^{(2)} \to \qq\llbracket x,y \rrbracket $ be the homomorphism such 
that $\rk\lp p\ud{1}{1}\rp = x$, $\rk\lp p\ud{1}{2}\rp = y$ and $\rk\lp p\ud{n}{i}\rp 
=0$ for $n>1$ and $i = 1,2$. Thus if $V$ is a representation of $S_n\times S_m$ then 
\begin{equation*}
  \rk(\ch_{m|n} (V)) = \frac{\dim V}{m! n!} x^m y^n \ .
\end{equation*}
Let $\qq \llbracket x,y \rrbracket_*$ be the power series of the form 
$\sum_{i=0}^{\infty}\sum_{j=0}^{\infty} a_{i,j} x^i y^j$ where $a_{2,0} \neq 0$ 
and $a_{i,j} = 0$ if $i<2$ or $i+j<3$. Let $\Lambda\ud{*}{2} = \rk^{-1}\qq \llbracket x,y
\rrbracket_*$. 

To define the partial Legendre transform we first define a variant of plethysm; 
$\circ_{(1)}$ which is an associative product on $\Lambda^{(2)}$: 
\begin{enumerate}
\item $f \mapsto f\circ_{(1)} g$ is a homomorphism $\Lambda^{(2)} \to \Lambda^{(2)}$,
      for any $g \in \Lambda^{(2)}$
\item $g \mapsto p\ud{n}{i} \circ_{(1)} g$ is a homomorphism $\Lambda^{(2)} \to
      \Lambda^{(2)}$,
\item $p\ud{n}{1} \circ_{(1)} p\ud{k}{i} = p\ud{nk}{i}$ and $p\ud{n}{2} \circ_{(1)}
      p\ud{k}{i} = p\ud{n}{2}$.
\end{enumerate}

For $f \in \Lambda\ud{*}{2}$ there is a unique $g \in \Lambda\ud{*}{2}$
satisfying the equation
\begin{equation} 
  g\circ_{(1)} \frac{\partial f}{\partial p\ud{1}{1}} + f = p\ud{1}{1}
  \frac{\partial f}{\partial p\ud{1}{1}} \ .
\end{equation}
For the existence and uniqueness we refer to Getzler and Kapranov \cite[Theorem~7.15]{GK}.
The proof in this case is completely analogous and goes through in almost the same way
without any subtleties. We call the function $g$ the partial Legendre transform of 
$f$ and denote it by $\cL^{(1)}f$.

A little bit of algebra shows that $\cL^{(1)}$ is an involution on $\Lambda\ud{*}{2}$,
that is $\cL^{(1)}\cL^{(1)}f = f$. We have the following result.

\begin{proposition} \label{prop:legendre}
  Let $\cW$ be an ungraded $\ss^2$ module such that $\cW(m,n) =0$ if $m<2$ or
  $m+n<3$. $F = e\ud{2}{1} - \ch(\cW)$ and $G = h\ud{2}{1} + \ch(\tT\cW)$ are
  elements of $\Lambda\ud{*}{2}$ and $G = \cL^{(1)}F$.
\end{proposition}

The proof of this proposition is essentially the same as the proof of Theorem~7.17 of 
\cite{GK}. One can also look at Theorem~5.8 of \cite{Cha1} for a different proof.

\section{Cohomology of the Interior}
\label{sec:interior}
In this section we study the cohomology of the interior $M_{0,m|n}$. It is easy to see 
that 
\begin{equation} \label{eq:interior}
  M_{0,m|n} \cong
  \left(\Big(\pp^{1}\Big)^{m+n}\backslash \left(
  \bigcup_{i=1}^m \bigcup_{j=i+1}^{m+n} \Delta_{i,j} \right)
  \right)/ \pgl(2,\cc) \ ,
\end{equation}
where $\pgl(2,\cc)$ acts diagonally and $\Delta_{i,j} = \{(z_1,\ldots,z_{m+n}) 
\mid z_i = z_j \}$.

\begin{proposition} \label{prop:dec}
  When $m\geq 3$, 
  \begin{equation*}
    H^*(M_{0,m|n}) \cong H^*(M_{0,m}) \otimes H^*(P_m^n)
  \end{equation*}
  where $P_m = \pp^1\backslash \{1,\ldots, m\}$ is the $m$ punctured projective plane.
  Moreover this decomposition respects the action of $S_m\times S_n$.
  
  The mixed Hodge structure on $H^i(M_{0,m|n})$ is pure of weight $2i$.
\end{proposition}

\begin{proof} 
  Consider the fiber bundle $M_{0,m|(n+1)} \to M_{0,m|n}$ with fiber $P_m$. $P_m$ is homotopic to a wedge of 
  circles, hence a one dimensional C-W complex. The fundamental group of the base acts trivially on the fibers, 
  (see Arnold \cite{Arn}), hence in the Leray spectral sequence associated to the fibration we have 
  \begin{equation*}
    E_{2}^{p,q} \cong H^p(M_{0,m|n}) \otimes H^q(P_m).
  \end{equation*}
  Moreover the fiber bundle has a section given by 
  \begin{equation*}
    z_{m+n+1} = \frac{z_1+ \ldots + z_m}{m} + 2 \Big( \max_{1 \leq k,l \leq m} |z_k-z_l| \Big) + 1.
  \end{equation*}
  It then follows that the only possible higher differential $d_2$ is trivial and we have $H^*(M_{0,m|(n+1)}
  \cong H^*(M_{0,m|n}) \otimes H^*(P_m)$. This completes the proof by induction on $n$. 
  
  The statement on the mixed Hodge structure of $H^i(M_{0,m|n})$ follows from the fact that $M_{0,m|n}$
  is isomorphic to a complement of hyperplanes in a projective spaces. This can be seen from 
  description \eqref{eq:interior}.    
\end{proof}

From the previous proposition it follows that the Poincar\'{e} polynomial of $M_{0,m|n}$,
$\cP_{M_{0,m|n}}(t)$, is the product of the Poincar\'{e} polynomials of $M_{0,m}$ and 
$P_m^n$. From Getzler \cite[Section~5.6]{Get} we know that $\cP_{M_{0,m}}(t) = (1-2t)(1-3t)\cdots (1- (m-2)t)$. It is easy to see that $\cP_{P_m^n}(t) = (1 - (m-1)t)^n$. Thus
\begin{equation*}
  \cP_{M_{0,m|n}}(t) = \big(1-(m-1)t\big)^n\prod_{k=2}^{m-2} (1- kt) \ .
\end{equation*}

Proposition \ref{prop:dec} gives a clear description of the $H^*(M_{0,m|n})$ as a 
representation of $S_m\times S_n$. 

First note that Getzler \cite{Get} determines completely the action of $S_m$ on 
$H^*(M_{0,m})$, whereas $S_n$ acts on it trivially.

Let $C_n$ be the vector space generated by the letters $\{x_1,\ldots, x_n \}$. $S_n$ 
acts on it by permuting the letters. $C_n$ is the direct sum of the standard 
representation and the trivial representation i.e. $C_n = V_{(n-1,1)}\oplus V_{(n)}$.

\begin{proposition} \label{prop:repPmn}
  We have the following description of the $S_m\times S_n$ action on the cohomology 
  of $P_m^n$.
  \begin{align*}
    H^k(P_m^n) & \cong 
    \lp \otimes^k V_{(m-1,1)} \rp \boxtimes (\wedge^k C_n)\\
    & \cong
    \left\{
    \begin{array}{l l}
      V_{(m)}\boxtimes V_{(n)} \ , & k=0 \\
      \lp \otimes^k V_{(m-1,1)} \rp \boxtimes 
        \lp V_{(n-k,1^k)} \oplus V_{(n-k+1,1^{k-1})} \rp \ , & 0<k< n \\
      \lp \otimes^n V_{(m-1,1)} \rp \boxtimes V_{(1^n)} \ , & k=n.
    \end{array}
    \right.
  \end{align*} 
\end{proposition}

\begin{proof}
  $S_m$ acts on $P_m$ by permuting the punctures, so $H^1(P_m)$ is the standard 
  representation $V_{m-1,1}$. 
  
  On the other hand $S_n$ acts on $P_m^{n}$ by permuting the 
  factors, thus $H^1(P_m^{n}) \cong V_{m-1,1}\boxtimes C_n$ as a representation of 
  $S_m \times S_n$. For $k>1$,  $H^k(P_m^n) \cong \lp \otimes^k H^1(P_m) \rp \boxtimes
  \lp \wedge^k C_n \rp $. 
  
  The second isomorphism follows form the decomposition $C_n = V_{(n-1,1)} \oplus 
  V_{(n)}$. Thus for $k<n$ we have $\wedge^k C_n = \left(\wedge^kV_{(n-1,1)}\right)\oplus 
  \left(V_{(n)}\otimes \wedge^{k-1} V_{(n-1,1)}\right)$ and $\wedge^n C_n = V_{(1^n)}$.
  Finally it is a fact that $\wedge^k V_{(n-1,1)} = V_{(n-k,1^k)}$.
\end{proof}

The Propositions \ref{prop:dec} and \ref{prop:repPmn} together give us the following 
decomposition for $m \geq 3$, 
\begin{align}
  H^{k}(M_{0,m|n}) & = \bigoplus_{l=0}^{k} H^{k-l}(M_{0,m}) \otimes H^l(P_m^n) 
  \nonumber \\
  & = \bigoplus_{l=0}^{k} \lp H^{k-l}(M_{0,m})\otimes \lp \otimes^l V_{(m-1,1)}
  \rp \rp \boxtimes \lp \wedge^l C_n \rp \ . \label{eq:rep}
\end{align}
In \eqref{eq:rep} we treat $H^{k-l}(M_{0,m})$ as just a representation of $S_m$. 

Hence we have the following relation for $m\geq 3$
\begin{align}
  H^{*}(M_{0,m|n}) & = \bigoplus_{l=0}^{n} \big(H^*(M_{0,m}) \otimes H^l(P_m^n) \big) \nonumber \\
                   & = \bigoplus_{l=0}^{n} \lp H^{*}(M_{0,m})\otimes \lp \otimes^l V_{(m-1,1)}
  \rp \rp \boxtimes \lp \wedge^l C_n \rp . \label{eq:rep1}
\end{align}

Let $\cM$ be the $\ss$ module 
\begin{equation} \label{eq:grav}
  \cM(n) = \left\{
  \begin{array}{l l}
    H^*(M_{0,n}) & n\geq 3, \\
    0 & n<3.
  \end{array} \right.
\end{equation}
Let $\fm_n = \ch_t(\cM(n))$ and $\fm = \sum_{n=1}^{\infty} \fm_n = \ch_t(\cM)$. 

Let $\cG$ be the $\ss^2$ module 
\begin{equation} \label{eq:int}
  \cG(m,n) = \left\{
  \begin{array}{l l}
    0 \ ,& m <2 \text{ or } m+n<3 \\
    H^*(M_{0,m|n}) \ , & \text{otherwise}
  \end{array} \right.
\end{equation}
Let $\D$ be the differential operator as in Proposition \ref{prop:operator}. From 
\eqref{eq:rep1} it follows that when $k\geq 3$
\begin{equation}
  \ch_t(\cG(k,n)) = \fm\ud{k}{1} s\ud{n}{2}- t\D\fm\ud{k}{1} \lp s\ud{n}{2} 
  + s\ud{n-1,1}{2} \rp+ \cdots + (-t)^n \lp \D^n \fm\ud{k}{1}\rp s\ud{1^n}{2} \ .
\end{equation}

$\bM_{0,2|n}$ are the Losev-Manin spaces and were extensively studied in \cite{LM}. Note
that 
\begin{equation*}
  M_{0,2|n} \cong \lp \cc^{\times} \rp^n / \cc^{\times}
\end{equation*}
where the quotient is taken under the diagonal action. It follows that $H^1(M_{0,2|n})
\cong V_{(1,1)} \boxtimes V_{(n-1,1)}$ (see \cite[Lemma~3.3]{BM2}). As before 
$H^k(M_{0,2|n}) \cong \wedge^k H^1(M_{0,2|n})$. Thus 
\begin{equation*}
  H^k(M_{0,2|n}) \cong \left\{
  \begin{array}{l l}
    V_{(2)}\boxtimes V_{(n-k,1^k)} & k<n \text{ even} \\
    V_{(1^2)}\boxtimes V_{(n-k,1^k)} & k<n \text{ odd} \\
    0 & k \geq n.
  \end{array} \right.
\end{equation*}

Hence it follows that 
\begin{equation}
  \ch_t(\cG(2,n)) = s\ud{2}{1}s\ud{n}{2} - ts\ud{1,1}{1}s\ud{n-1,1}{2} + \ldots = 
  \sum_{k-0}^{n-1} (-t)^k \lp \D^k s\ud{2}{1} \rp s\ud{n-k,1^k}{2} \ .
\end{equation}

Adding up $\ch_t(\cG(k,l))$ for all $k,l$ we get

\begin{align} \label{eq:formula}
  \ch_t(\cG)  = \ & \fm\ud{}{1} + 
  \lp \fm\ud{}{1}+s\ud{2}{1} \rp\sum_{n=1}^{\infty} s\ud{n}{2}  \nonumber \\
  & + \sum_{k=1}^{\infty} (-t)^k \lp 
  \D^k\fm\ud{}{1}\sum_{n=k}^{\infty} s\ud{n-k+1,1^{k-1}}{2} 
  + \D^k \lp \fm\ud{}{1}+s\ud{2}{1} \rp\sum_{n=k+1}^{\infty} s\ud{n-k,1^k}{2} \rp .
\end{align}

\section{Cohomology of $\bM_{0,m|n}$}
\label{sec:compactification}
In this section we shall determine the action of $S_m\times S_n$ on the cohomology of 
$\bM_{0,m|n}$. To do this let us introduce the $\ss^2$ module $\cW$,
\begin{equation} \label{eq:comp}
  \cH(m,n) = \left\{
  \begin{array}{l l}
    0 \ ,  & m < 2 \text{ or } m+n < 3 \\
    H^*(\bM_{0,m|n})  \ ,& \text{otherwise.}
  \end{array}
  \right.
\end{equation}
We shall derive a formula relating $\ch_t(\cG)$ (see \eqref{eq:grav}) and $\ch_t(\cH)$ using 
the partial Legendre transform.

Let $X$ be an algebraic variety and $\emptyset \subset X_0 \subset \ldots \subset X_n 
= X$ a filtration on it by closed subvarieties $X_p \subset X$. Then there is a spectral 
sequence in cohomology with compact support (see Petersen \cite[Section~1]{Pet}),
\begin{equation*}
  E_{1}^{p,q} =  H_c^{p+q}(X_p \backslash X_{p-1}) \Longrightarrow H^{p+q}_c(X) .
\end{equation*}
The differentials of this spectral sequence are compatible with the mixed Hodge 
structures. Further if a finite group $G$ acts on $X$ and keeps each $X_p$ invariant 
then $E_j^{p,q} $ has an action of $G$ and the differentials $d_j$ are $G$ 
equivariant.

In our situation let $X = \bM_{0,m|n}$ and $X_p$ be the union of all strata of dimension 
at most $p$
\begin{equation*}
  X_p  = \bigsqcup_{\substack{T \in  \tT(m,n) \\ |E(T)| = m+n-3-p}} 
  \bM(T).
\end{equation*}
Clearly
\begin{equation*}
  X_p \backslash X_{p-1} = \bigsqcup_{\substack{T \in  \tT(m,n) \\ |E(T)| = m+n-3-p}} 
  M(T) .
\end{equation*} 
Thus
\begin{equation} \label{eq:ss}
  E_1^{p,q} = \bigoplus_{\substack{T \in  \tT(m,n) \\ |E(T)| = m+n-3-p}} 
  H_c^{p+q}(M(T)) . 
\end{equation}

From Proposition \ref{prop:dec} and Poincar\'{e} duality it follows that the mixed Hodge 
structure on $H^i_c(M_{0,m|n})$ is pure of weight $2(i-m-n+3)$. This 
implies that $E_1^{p,q}$ has a pure Hodge structure of weight $2q$. Hence the Spectral 
sequence collapses in the $E_2$ page
\begin{equation*}
  E_{2}^{p,q} \cong E_{\infty}^{p,q}. 
\end{equation*}

Moreover, from the fact that all the cohomology of $\bM_{0,m|n}$ is algebraic it follows 
that 
\begin{equation}
  E_2^{p,p} \cong H^{2p}(\bM_{0,m|n}) \quad \text{and} \quad E_2^{p,q} = 0 
  \text{ if } p \neq q \ .
\end{equation}
Thus there is a resolution 
\begin{equation} \label{eq:res}
  H^{2p}(\bM_{0,m|n}) \to \bigoplus_{\substack{T \in  \tT(m,n) \\ |E(T)| = m+n-3-p}} 
  H_c^{2p}(M(T)) \to \cdots \to H^{m+n-3+p}_c(M_{0,m|n}) \ .
\end{equation}

\begin{theorem} \label{theorem:legendre}
  Let 
  \begin{equation*}
    F = t^{-6} \ch_t(\cG) 
    \Big|_{ \frac{t \mapsto t^{-2}}{p\ud{n}{i} \mapsto t^{2n}p\ud{n}{i}} } \ ,
  \end{equation*}
  then $h\ud{2}{1} + \ch_t(\cH) = \cL^{(1)}\lp e\ud{2}{1} - F \rp$. 
\end{theorem}

\begin{remark}
  Note that $(e\ud{2}{1}-F) \in \Lambda\ud{*}{2}\llbracket t \rrbracket$. More over 
  $\circ_{(1)}$ extends to $\Lambda^{(2)}\llbracket t \rrbracket$ in a natural way: 
  $p\ud{n}{1}\circ_{(1)} t = t^n$ and $p\ud{n}{2}\circ_{(1)} t =  p\ud{n}{2}$. Hence 
  $\cL^{(1)}$ makes sense on $\Lambda\ud{*}{2}\llbracket t \rrbracket$. 
\end{remark}

\begin{proof}  
  As in \cite[section 5.8]{Get} we shall consider (graded) $\ss^2$ modules $\cV$ with a 
  further $\zz/2$-grading $\cV = \cV_{(0)} \oplus \cV_{(1)}$. In this case define 
  \begin{equation*}
    \ch_t(\cV) = \ch_t(\cV_{(0)}) - \ch_t(\cV_{(1)}) \ .
  \end{equation*}  
    
  By Poincar\'{e} duality $H^k_c(M_{0,m|n}) \cong H^{2m+2n-6-k}(M_{0,m|n})^{\vee} 
  \otimes \cc(-m-n+3)$ where $\cc(-\ell)$ is the $\ell$-fold tensor power of the dual of the 
  Tate Hodge structure. Thus $H^k_c(M_{0,m|n})$ has pure Hodge structure of weight 
  $2(k-m-n+3)$.
  
  Define the $\zz/2$-graded $\ss^2$ module $\cV$ as follows
  \begin{equation*}
    \cV(m,n) = 0 \text{ if } m<2 \text{ or } m+n<3 \ ,
  \end{equation*}
  otherwise 
  \begin{align*}
    \cV_{(0)}(m,n) & = \bigoplus_{k=0}^{\infty} H^{2k}_c(M_{0,m|n}) \\
    \cV_{(1)}(m,n) & = \bigoplus_{k=0}^{\infty} H^{2k+1}_c(M_{0,m|n}).
  \end{align*}
  Here we consider $H^k_c(M_{0,m|n})$ with weight grading for the mixed Hodge structure on it, 
  that is $H^k_c(M_{0,m|n})$ is the $2(k-m-n+3)$ graded component. Then 
  \begin{equation*}
    F =  \ch_t(\cV) \ .
  \end{equation*}
  The construction $\tT$ of Section~\ref{sec:smod} extends naturally to $\zz/2$-graded 
  $\ss^2$ modules (tensor product of odd and odd is even, even and even is even where as
  that of odd and even is odd). Proposition~\ref{prop:legendre} generalises to the case of
  $\zz/2$-graded $\ss^2$ modules.
 
  If we add up all the terms of the spectral sequence \eqref{eq:ss} placing 
  $E^{p,q}_1$ in bi-degree $2q, (p+q) \bmod{2}$ we get the graded vector space $\tT\cV(m,n)$. 
  The differential $d_1: E_1^{p,q} \to E_1^{p+1,q}$ gives a differential on $\tT\cV(m,n)$ 
  and the resolution \eqref{eq:res} shows that $H^*(\bM_{0,m|n})$ is the homology of
  of the complex $(\tT\cV(m,n),d_1)$.
  
  Hence $\ch_t(\bM_{0,m|n}) = \ch_t(\tT \cV(m,n))$. This completes the proof.
\end{proof}

\appendix

\section{Calculations}
Recall the $\ss^2$ modules $\cG$ from \eqref{eq:int} and $\cH$ from \eqref{eq:comp}.
In this section we compute the first few terms of the characteristics of $\cG$ and
$\cH$ and list them in Table~\ref{table:int} and Table~\ref{table:comp} respectively.
 
Formula \eqref{eq:formula} gives a recipe for calculating $\ch_t(\cG)$ from 
$\ch_t(\cM)$. The $\ss$ module $\cM$ was defined in $\eqref{eq:grav}$. The action of
$S_n$ on $H^*(M_{0,n})$ was calculated in Getzler \cite{Get} (see Theorem 5.7). Let 
$\mu$ be the M\"{o}bius function, and let $ \displaystyle R_n(t) = (1/n) 
\sum_{d \text{ divides } n} \mu(n/d)/ t^d $. Further let $\kappa$ be the linear 
operator on $\Lambda$ which is 0 on the 0th,1st and 2nd graded components of $\Lambda$ 
and identity on the rest. Then 
\begin{equation*}
  \ch_t(\cM) = \kappa \lp \frac{1+tp_1}{1-t^2} \prod_{n=1}^{\infty} (1+t^np_n)^{R_n(t)}
   \rp .
\end{equation*}
Thus using \eqref{eq:formula} we can calculate the first few terms of $\ch_t(\cG)$. Of 
course $\ch_t(\cG(k,n))$ starts to be interesting when $k \geq 3$ and $n \geq 2$. In table 
Table~\ref{table:int} we list these terms for $k+n \leq 6$. 

Using Theorem~\ref{theorem:legendre} we can in principle determine $\ch_t(\cH)$. The 
theorem gives a fixed point formula and first few terms of $\ch_t(\cH)$ can be obtained 
from $\ch_t(\cG)$ by performing several iterations. Again the terms $\ch_t(\cH(k,n))$
for $n=1$ can be easily computed and are uninteresting. We list the terms corresponding
to $k+n \leq 6$ and $n>1$ in Table~\ref{table:comp}.

The calculations involving symmetric functions were done using the Maple package SF
\cite{SF} by Stembridge.

\begin{landscape}

\begin{center}

\begin{table} 
\begin{equation*}
  \renewcommand{\arraystretch}{2}
  \begin{array}{|l|l|} \hline
    (m,n) & \ch_t \lp H^*(M_{0,m|n}) \rp \\ \hline
    
    (3,2) & s\ud{3}{1}s\ud{2}{2} -ts\ud{2,1}{1}\lp s\ud{2}{2}+s\ud{1^2}{2} \rp + 
            t^2\lp s\ud{3}{1}+s\ud{2,1}{1}+s\ud{1^3}{1}\rp s\ud{1^2}{2} \\ \hline
            
    (3,3) & s\ud{3}{1}s\ud{3}{2} - ts\ud{2,1}{1}\lp s\ud{3}{2} + s\ud{2,1}{2} \rp 
            + t^2\lp s\ud{3}{1}+ s\ud{2,1}{1}+s\ud{1^3}{1} \rp \lp s\ud{2,1}{2} 
            +s\ud{1^3}{2} \rp  - t^3\lp s\ud{3}{1}+3s\ud{2,1}{2}+s\ud{1^3}{2} \rp
            s\ud{1^3}{2} \\ \hline           
             
    (4,2) & s\ud{4}{1}s\ud{2}{2} - t\lp s\ud{2^2}{1}s\ud{2}{2} + s\ud{3,1}{1}\lp 
            s\ud{2}{2} + s\ud{1^2}{2} \rp \rp + t^2\lp \lp s\ud{4}{1}+s\ud{2^2}{1}\rp 
            s\ud{1^2}{2} + \lp s\ud{3,1}{1}+s\ud{2,1^2}{1} \rp \lp s\ud{2}{2} + 
            2s\ud{1^2}{2} \rp \rp \\ 
          & - t^3 \lp s\ud{4}{1} + 2s\ud{2,2}{1} +2s\ud{3,1}{1}
            + 2s\ud{2,1^2}{1} + s\ud{1^4}{1} \rp s\ud{1^2}{2} \\ \hline
  \end{array} 
\end{equation*}
\caption{Equivariant Poincar\'{e} polynomial for the interior}
\label{table:int}
\end{table}

\begin{table}
\begin{equation*}
  \renewcommand{\arraystretch}{2}
  \begin{array}{|l|l|l|} \hline
    (m,n) & \ch_t \lp H^*(\bM_{0,m|n}) \rp & \text{Poincar\'{e} Polynomial}  \\ \hline
    
    (2,2) & (1+t^2)s\ud{2}{1}s\ud{2}{2} & 1+t^2  \\ \hline
    
    (2,3) &  (1+t^4)s\ud{2}{1}s\ud{3}{2} + 
             t^2 \lp s\ud{2}{1} \lp s\ud{3}{2} + s\ud{2,1}{2} \rp + s\ud{1^2}{1}s\ud{3}{2} \rp 
             & 1 +4t^2 + t^4 \\ \hline
             
    (3,2) &  (1+t^4) s\ud{3}{1}s\ud{2}{2} + 
             t^2 \lp s\ud{3}{1} \lp 2 s\ud{2}{2} + s\ud{1^2}{2} \rp + s\ud{2,1}{1}s\ud{2}{2} \rp 
             & 1 +5t^2 + t^4\\ \hline             
             
    (2,4) &  (1+t^6) s\ud{2}{1}s\ud{4}{2} + 
             (t^2+t^4) \lp s\ud{2}{1} \lp 2s\ud{4}{2} + s\ud{3,1}{2} s\ud{2^2}{2} \rp 
             + s\ud{1^2}{1} \lp s\ud{4}{2} + s\ud{3,1}{2} \rp \rp 
             & 1 + 11t^2+11t^4 + t^6 \\ \hline             
             
    (3,3) &  (1+t^6) s\ud{3}{1}s\ud{3}{2} + 
             (t^2+t^4) \lp s\ud{3}{1} \lp 3 s\ud{3}{2} + 2 s\ud{2,1}{2} \rp 
             + s\ud{2,1}{1} \lp 2 s\ud{3}{2} + s\ud{2,1}{2} \rp \rp 
             & 1+ 15t^2+15t^4 +t^6\\ \hline
             
    (4,2) &  (1+t^6) s\ud{4}{1} s\ud{2}{2} + 
             (t^2+t^4) \lp s\ud{4}{1} \lp 4 s\ud{2}{2} + s\ud{1^2}{2} \rp 
             + s\ud{3,1}{1} \lp 2 s\ud{2}{2} + s\ud{1^2}{2} \rp + s\ud{2,2}{1}s\ud{2}{2} \rp
             & 1+16t^2+16t^4 +t^6\\ \hline
  \end{array}
\end{equation*}
\caption{Equivariant Poincar\'{e} polynomial of $\bM_{0,m|n}$}
\label{table:comp}
\end{table}

\end{center}

\end{landscape}

\end{document}